\documentclass[leqno]{amsart}
\usepackage{amssymb}
\usepackage{mathrsfs}
\usepackage{amsmath, amsfonts, vmargin, enumerate}
\usepackage{graphics}
\usepackage{verbatim}
\usepackage{amsthm}
\usepackage{latexsym,bm}
\usepackage{euscript}

\input xypic
\xyoption {all}

 \makeatletter

\newtheorem{thm}{Theorem}[section]
\newtheorem{prop}[thm]{Proposition}
\newtheorem{cor}[thm]{Corollary}
\newtheorem{lem}[thm]{Lemma}
\newtheorem{defi}[thm]{Definition}
\newtheorem{remark}[thm]{Remark}
\newtheorem{example}[thm]{Example}
\newtheorem{pb}[thm]{Problem}

\numberwithin{equation}{section}

\newcommand{\real}{{\mathbb R}}

\newcommand{\com}{{\mathbb C}}
\newcommand{\un}{1\mkern -4mu{\textrm l}}

\newcommand{\M}{{\mathcal M}}

\newcommand{\e}{\varepsilon}

\newcommand{\8}{\infty}

\newcommand{\ra}{\rightarrow}

\newcommand{\n}{\noindent}

\newcommand{\be}{\begin{eqnarray*}}
\newcommand{\ee}{\end{eqnarray*}}
\newcommand{\beq}{\begin{equation}}
\newcommand{\eeq}{\end{equation}}
\newcommand{\beqn}{\begin{equation*}}
\newcommand{\eeqn}{\end{equation*}}

\begin{document}

\title{Noncommutative harmonic analysis on semigroup and ultracontractivity}

\thanks{{\it Key words:} Semigroups, noncommutative $L_p$-spaces, ultracontractivity, Sobolev embedding inequalities, logarithmic Sobolev inequality, noncommutative maximal inequalities, spectral multipliers.}

\author{Xiao Xiong}
\address{Department of Mathematical Sciences, Seoul National University, Gwanak-ro 1, Gwanak-gu,
Seoul 151-747, Republic of Korea}
\email{xxiong@snu.ac.kr}

\date{}
\maketitle

\markboth{X. Xiong}%
{Noncommutative harmonic analysis and ultracontractivity}

\begin{abstract}


We extend some classical results of Cowling and Meda to the noncommutative setting. Let $(T_t)_{t>0}$ be a symmetric contraction semigroup on a noncommutative space $L_p(\M),$ and let the functions $\phi$ and $\psi$ be regularly related. We prove that the semigroup $(T_t)_{t>0}$ is $\phi$-ultracontractive, i.e. $\|T_t x\|_\8 \leq C \phi(t)^{-1} \|x\|_1$ for all $x\in L_1(\M)$ and $ t>0$ if and only if its infinitesimal generator $L$ has the Sobolev embedding properties: $\|\psi(L)^{-\alpha} x\|_q \leq C'\|x\|_p$ for all $x\in L_p(\M),$ where $1<p<q<\8$ and $\alpha =\frac{1}{p}-\frac{1}{q}.$ We establish some noncommutative spectral multiplier theorems and maximal function estimates for generator of $\phi$-ultracontractive semigroup. We also show the equivalence between $\phi$-ultracontractivity and logarithmic Sobolev inequality for some special $\phi$. Finally, we gives some results on local ultracontractivity.


\end{abstract}


\section{Introduction and Main Result}\label{Intro}

Let $L$ be a densely defined positive operator on $L_2(\Omega),$ where $\Omega$ is a $\sigma$-finite measure space. Suppose that $\{P_\lambda\}$ is the spectral resolution of $L$:
$$Lf=\int_0^\8 \lambda \;dP_\lambda f \;,\;\;\;\forall f\in \rm{Dom}(L).$$
If $m$ is a bounded function on $[0,\8),$ then by the spectral theorem, the multiplier operator $m(L)$ defined by
$$m(L)f=\int_0^\8 m(\lambda)\; dP_\lambda f\;,\;\;\;\forall f\in L_2(\Omega),$$
is bounded on $L_2(\Omega).$ Let $(T_t)_{t>0}=(e^{-tL})_{t>0}$ be the operator semigroup generated by $L$ and assume that each $T_t$ has the contraction property:
$$\|T_t f\|_p\leq \|f\|_p\;,\;\;\;\forall f\in L_2(\Omega)\cap L_p(\Omega)$$
whenever $1\leq p\leq \8.$ A semigroup $(T_t)_{t>0}$ with the above properties is called a symmetric contraction semigroup. Stein \cite{Stein1970} developed a Littlewood-Paley theory for such semigroups, with some additional hypotheses.
Coifman and Weiss \cite{CW1976}, Cowling \cite{Cowling1983} presented an alternative and simpler approach to obtain multiplier results and maximal inequalities.

A symmetric contraction semigroup $(T_t)_{t>0}$ that maps $L_1(\Omega)$ into $L_\8 (\Omega)$ for every positive $t$ is said to be ultracontractive. Letting the functions $\phi$ and $\psi$ be regularly related (see below for the precise definition), Cowling and Meda \cite{CM1993} showed that $(T_t)_{t>0}$ is $\phi$-ultracontractive, i.e.
$$\|T_tf\|_\8\leq C \phi(t)^{-1}\|f\|_1\;\;\;\forall f\in L_1(\Omega)\;\;\forall t\in \real^+,$$
if and only if the infinitesimal generator $L$ has Sobolev embedding properties, namely,
$$\|\psi(L)^{-\alpha}f\|_q\leq C'\|f\|_p\;\;\;\forall f\in L_p(\Omega)$$
whenever $1<p<q<\8$ and $\alpha=\frac{1}{p}-\frac{1}{q}.$

Recently, more attention has been turned to symmetric contractive semigroups on noncommutative spaces $L_p(\M)$, associated to $(\M,\tau)$, where $\M$ is a von Neumann algebra with a normal finite faithful trace $\tau.$ Diffusion semigroups on noncommutative $L_p$-spaces were considered for the first time in \cite{JX2002,JX2007}, where the authors developed noncommutative maximal inequalities for such semigroups. Based on these, \cite{JLX2006} studied a noncommutative generalization of Stein's diffusion semigroups considered in \cite{Stein1970}. Then, \cite{JM2010} studied noncommutative Riesz transform associated to a semigroup $(T_t)_{t>0}$ of completely positive maps on a von Neumann algebra with negative generator. BMO and tent spaces associated with semigroups were developed in \cite{Mei2008,JM2011}. We also refer the reader to \cite{Kriegler2011} for the result that semigroups on noncommutative $L_p$-spaces have extension under suitable conditions, similar to those on commutative $L_p$-spaces.

\medskip
In this paper, we study the ultracontractivity of semigroups $(T_t)_{t>0}$ acting on noncommutative $L_p$-spaces. We show that the $\phi$-ultracontractivity of $(T_t)_{t>0}$ is equivalent to the Sobolev embedding properties of its generator $L,$ and establish some spectral multiplier theorems and maximal inequalities. This paper can be viewed as the noncommutative extension of \cite{CM1993}.

To state our main results, some notation and basic knowledge should be in order. Throughout this paper, $\M$ is a von Neumann algebra with a normal finite faithful trace $\tau.$ Let $S^+_{\M}$ be the set of all positive elements $x$
in $\M$ with $\tau(s(x))<\infty$, where $s(x)$ denotes the support of $x$, i.e.,
the smallest projection $e$ such that $exe=x$. Let
$S_{\M}$ be the linear span of
$S^+_{\M}$. Then every $x\in
S_{\M}$ has finite trace, and
$S_{\M}$ is a w*-dense $*$-subalgebra of
$\M$. Let $0< p<\infty$. For any $x\in S_{\M}$, the
operator $|x|^p$ belongs to $S^+_{\M}$
(recalling $|x|=(x^*x)^{\frac{1}{2}}$). We define
$$\|x\|_p=\big(\tau(|x|^p)\big)^{\frac{1}{p}}.$$
One can check that $\|\cdot\|_p$ is a norm on
$S_{\M}$ if $1\leq p<\8$ and is a $p$-norm if $0<p<1$. The completion of
$(S_{\M},\|\cdot\|_p)$ is denoted by $L_p(\M)$,
which is the usual noncommutative $L_p$-space associated to
$(\M,\tau)$. For convenience, we set $L_{\infty}(\M)=\M$
equipped with the operator norm $\|\cdot\|_{\M}$.

We denote by $L_0(\M,\tau)$ the family of all measurable operators. For such an operator $x$, we define
$$\lambda_s(x)=\tau(e^\perp_s(|x|)),\;\;s>0$$
where $e^\perp_s(x)=\un_{(s,\infty)}(x)$ is the spectral projection of $x$ corresponding to the interval $(s,\8),$ and
$$\mu_t(x)=\inf \{s>0: \lambda_s(x)<t\},\;\;t>0.$$
The function $s \mapsto \lambda_s(x)$ is called the distribution function of $x$ and $\mu_t(x)$ the generalized singular numbers of $x$. Similarly to the classical case, for $0<p<\8, 0<q\leq \8,$ the noncommutative Lorentz space $L_{p,q}(\M)$ is defined to be the collection of all measurable operators $x$ such that
$$\|x\|_{p,q}=\big(\int_0^\8 (t^{\frac1p} \mu_t(x))^q \frac{dt}{t}\big)^{\frac1q} <\8.$$
Clearly, $L_{p,p}(\M)=L_{p}(\M)$. The space $L_{p,\8}(\M)$ is usually called a weak $L_p$-space and
$$\|x\|_{p,\8}=\sup_{s>0}s\lambda_s(x)^{\frac1p}.$$
Noncommutative $L_p$-spaces also behave well with respect to interpolation. Let $1\le p_0<p_1\leq \infty$, $1\le q\le \8$ and $0<\eta<1$. Then
 \be
 \big( L_{p_0}(\M),\, L_{p_1}(\M) \big)_{\eta}= L_{p}(\M)\;\text{ and }\;
\big( L_{p_0}(\M),\, L_{p_1}(\M) \big)_{\eta, q}= L_{p,q}(\M),
\ee
where $\frac{1}{p}=\frac{1-\eta}{p_0}+\frac{\eta}{p_1}$. The reader is referred to \cite{PX2003} and  \cite{Xu2007} for more information on noncommutative
$L_p$-spaces.

\medskip

Let $T:\;\M\rightarrow \M$ be a linear map satisfying
\begin{enumerate}[{\rm(i)}]
\item $T$ is a contraction on $\M$: $\|Tx\|_{\8}\leq \|x\|_{\8}\;,$
\item $\tau\circ T\leq \tau:$ for all $x\in L_1(\M)\cap\M_+$, $\tau(Tx)\leq \tau(x)$,
\item $T$ is positive: $x>0$ implies $Tx>0$,
\item $T$ is symmetric relative to $\tau: \tau(T(y)^*x)= \tau(y^*T(x))$ for all $x,y\in L_2(\M)\cap\M\;.$
\end{enumerate}
Properties (i), (ii) and (iii) imply $\|Tx\|_p\leq \|\M\|_p$ for every $1\leq p< \8$ (see\cite[Lemma~1.1]{JX2007}). Thus, $T$ naturally extends to a contraction on $L_p(\M)$ for every $1\leq p \leq \8$. In this paper, we are interested in the symmetric positive contraction semigroup $(T_t)_{t>0}$. This means that for each $t>0$ the operator $T_t$ satisfies the above (i)-(iv) and that for any
$x \in \M$, $T_t(x) \rightarrow x$ in the $w^*$-topology of $\M$ when $t \rightarrow 0^+.$ Such a semigroup admits an infinitesimal (negative) generator $L$, i.e., $T_t = e^{-tL}$. See \cite{JLX2006,JX2007} for more details.

\medskip

The norm of a bounded linear operator $T$ on a noncommutative $L_p$-space $L_p(\M)$ or Lorentz space $L_{p,q}(\M)$ will be denoted by $|||T|||_p$ or $|||T|||_{p,q}$ respectively. An operator which is bounded from $L_p(\M)$ to $L_q(\M)$ or from $L_p(\M)$ to $L_{q,\8}(\M)$ is said to be of strong type $(p,q)$ or weak type $(p,q)$ respectively. The corresponding norms are denoted by $|||T|||_{p\ra q}$ and $|||T|||_{p\ra (q,\8)}$ respectively. The sector $\{z\in \mathbb{C} : |\arg(z)|<\omega\}$ is denoted by $\Gamma_\omega,$ and the half-plane $\Gamma_{\frac  \pi 2}$ by $\Gamma$ simply. We denote $H^\8(D)$ the space of bounded holomorphic functions on the domain $D.$ Constants will be denoted by letters $A$, $C$ etc., sometimes with subscripts to indicate which the constants depend on.

\medskip

Now we state the main result of this paper. Given a semigroup $T_t$ and a positive function $\phi$ on $\real^+,$ we say that $T_t$ is $\phi$-ultracontractive if there exists a constant $A$ such that $\|T_tx\|_\8\leq A\phi(t)^{-1}\|x\|_1$ holds for all $t>0$. Our main result concerns how such ultracontractive estimates may be related to the Sobolev inequalities of the infinitesimal generator of $T_t.$ Following \cite{CM1993}, we say that a pair of functions $\phi:\real^+\rightarrow \real^+$ and $\psi:\Gamma \rightarrow \mathbb{C}$ is regularly related if
\begin{enumerate}[{\rm(1)}]
\item $\phi$ is increasing, continuous and surjective;
\item $\phi$ satisfies the $\Delta_2$ condition: there exists a constant $C_{\phi}$ such that
   $$\phi(2t)\leq C_{\phi} \phi(t), \;\;\forall t \in \real^+;$$
\item for all $\theta\in \real^+,$ there exists a constant $D_{\phi,\theta}$ such that
     $$\int_0^t \phi(s)^\theta \frac{ds}{s} \leq D_{\phi,\theta} \phi(t)^\theta,\;\;\forall t\in \real^+$$ and
     $$\int_t^\8 \phi(s)^{-\theta} \frac{ds}{s} \leq D_{\phi,\theta} \phi(t)^{-\theta},\;\;\forall t\in \real^+;$$
\item $\psi$ is holomorphic in $\Gamma$, and $\psi(\overline{z})=\overline{\psi(z)}$ for all $z\in \Gamma;$
\item for all $\omega\in (0,\frac{\pi}{2}),$ there exists a constant $C_{\psi,\omega}$ such that
      $$\inf\{|\psi(z)|:z\in \Gamma_{R,\omega} \} \geq C_{\psi,\omega}\sup \{|\psi(z)|: z\in \Gamma_{R,\omega}\}, \;\;\forall\; R\in\real^+ $$ where $\Gamma_{R,\omega}=\{z\in \mathbb{C}: R<{\rm{Re}}(z)<2R, |\arg(z)|<\omega\};$
\item for all $\theta\in \real^+,$ there exists a real-valued measurable function $\phi_\theta,$ equivalent to $\phi^\theta$ in the sense that the function $ \big|\frac{\phi(\cdot)^\theta }{ \phi_\theta (\cdot)}\big|$ is bounded and bounded away from $0$ on $ \real^+,$ with the property that, if $$\psi_\theta (z)= \int_0^\8 e^{-zt}\phi_\theta (t)\frac{dt}{t},\;\;\forall z\in \Gamma,$$ then $\psi_\theta$ is equivalent to $\psi^{-\theta}$ in the sense that the function $ |\psi(\cdot)^\theta \psi_\theta (\cdot)|$ is bounded and bounded away from $0$ in any proper sub-sector $\Gamma_\omega \subset \Gamma.$

\end{enumerate}
Some examples of such pairs are given in \cite{CM1993}:
\begin{enumerate}[{\rm(a)}]
\item $\phi(t)=t^\alpha$ and $\psi(z)=z^\alpha$ for positive $\alpha;$
\item $\phi(t)=t^\alpha(1+t)^{\beta-\alpha}$ and $\psi(z)=z^\beta(1+z)^{\alpha-\beta}$ for positive $\alpha,\beta\in \real^+;$
\item $\phi(t)=t^\alpha\log(2+t)^\beta$ and $\psi(z)=z^\alpha\log(2+\frac{1}{z})^{-\beta}$ for positive $\alpha,\beta\in \real^+.$
\end{enumerate}
The following is our main theorem about ultracontractivity.

\begin{thm}\label{phipsiultra}
Let $\phi$ and $\psi$ be a pair of regularly related functions. Then the following properties are equivalent:
\begin{enumerate}[{\rm(i)}]
\item $\|T_tx\|_\8\leq A\phi(t)^{-1}\|x\|_1,\;\forall t\in\real^+,\;\forall x\in L_1(\M);$

\item for all $p,q$ and $\alpha$ such that $1\leq p<q\leq \8,$ and $\alpha=\frac{1}{p}-\frac{1}{q}$,
     $$\|T_tx\|_q\leq A'\phi(t)^{-\alpha}\|x\|_p,\;\forall t\in\real^+,\;\forall x\in L_p(\M);$$
\item there exist $p,q$ and $\alpha$ such that $1\leq p<q\leq \8,$ and $\alpha=\frac{1}{p}-\frac{1}{q}$,
     $$\|T_tx\|_q\leq A'\phi(t)^{-\alpha}\|x\|_p,\;\forall t\in\real^+,\;\forall x\in L_p(\M);$$
\item for all $r$ and $\alpha$ such that $1<r<\8$ and $\alpha=1-\frac{1}{r},$
     $$\|\psi(L)^{-\alpha}x\|_{r,\8}\leq C_r\|x\|_1,\;\forall x\in L_1(\M);$$
\item for all $p,q$ and $\alpha$ such that $1<p<q<\8$ and $\alpha=\frac{1}{p}-\frac{1}{q},$
     $$\|\psi(L)^{-\alpha}x\|_q\leq C_{p,q}\|x\|_p,\;\forall x\in L_p(\M);$$
\item there exist $p,q$ and $\alpha$ such that $1<p<q<\8$ and $\alpha=\frac{1}{p}-\frac{1}{q},$
     $$\|\psi(L)^{-\alpha}x\|_q\leq C \|x\|_p,\;\forall x\in L_p(\M);$$
\item there exist $p,q$ and $\alpha$ such that $1<p<q<\8$ and $\alpha=\frac{1}{p}-\frac{1}{q},$
     $$\|\psi(L)^{-\alpha}x\|_{q,\8}\leq C \|x\|_p,\;\forall x\in L_p(\M).$$
\end{enumerate}

\end{thm}

The above theorem will be proved in section 2. In section 3, we will establish multiplier and maximal operator results for such ultracontactive semigroups. Let $m$ be a complex-valued measurable function on $\real^+$. Recall that the multiplier operator is defined by
$$m(L)x = \int_0^\8 m(\lambda)\;dP_\lambda x,$$
for $x$ in some appropriate domain. We will give sufficient conditions on $m$ to ensure the boundedness of $m(L)$ from $L_p(\M)$ to $L_q(\M)$ (or weak $L_q$). A H\"{o}rmander-like multiplier result will also be established for ultracontractive semigroup whose generator $L$ satisfies
 $$|||L^{iu}|||_p\leq C_p (1+|u|)^{\sigma |\frac 1 p-\frac 1 2|},\;\;\forall\; u\in \real,$$
or its weak form. Finally in this section, we will investigate the maximal operator $\sup_{t>0}|t^\alpha m(tL)|$ acting on $\M$, and also give some additional hypothesis to make it bounded of strong (or weak) type $(p,q)$.

In section \ref{log-Sb}, we will show the equivalence between $\phi$-ultracontractivity and the following logarithmic Sobolov inequality for the infinitesimal generator $L$:
\be
\tau(x^2 \log x) \leq \e \tau(xLx) +C M(\e)\|x\|_2^2+\|x\|_2^2\log \|x\|_2\,.
\ee
Here $\phi(t)=e^{-M(t)}$ with some monotonically decreasing function $M(t)$. In the last section we give some results for local ultracontractivity.


\section{Proof of the Main Theorem}


We recall one of the principal results of Jiao and Wang \cite{JW2014}, which extends the main result of \cite{Cowling1983} to the noncommutative case, and will be a main tool in this section. Suppose that $1<p<\8$, and $m$ is a bounded function on $[0,\8)$. If $\omega>\pi |\frac{1}{p}-\frac{1}{2}|,$ and $m$ extends to an $H^\8 (\Gamma_\omega)$-function, denoted $m_\omega,$ then $m(L)$ is of strong type $(p,p)$, and
\beq\label{mstrongpp}
|||m(L)|||_p\leq C_{p,\omega}\|m_\omega\|_\8.
\eeq
Applying the real interpolation method, we can easily deduce that $m(L)$ is also bounded on $L_{q,\8}(\M)$ for $q$ lies between $2$ and $p,$ and that
\beq\label{mweakqq}
|||m(L)|||_{q,\8}\leq C_{q,\omega}'\|m_\omega\|_\8.
\eeq
Like in \cite{JW2014}, we also suppose that the spectral projection $P_0$ onto the kernel of $L$ is trivial on $L_p(\M)$ in the sequel.

Let us turn to proving Theorem \ref{phipsiultra}. This theorem has been proved by Junge and Mei in \cite{JM2010} for the special pair of functions $(t^{\frac n 2}, z^{\frac n 2} ).$ The equivalence of (i), (ii) and (iii), ensured by interpolation and duality in \cite{JM2010}, holds for general $\phi$ as well; the proof is exactly the same as the one in the classical case \cite{CM1993}. To show their relation with the Sobolev inequalities of $\psi(L)^{-\alpha}$, we will establish the implications (i)$\Rightarrow$(iv) and (vii)$\Rightarrow$(iii), which are entitled Lemmas \ref{iimpliesiv} and \ref{viiimpliesiii} below respectively.

\begin{lem}\label{iimpliesiv}
Suppose that $(T_t)_{t>0}$ is a semigroup with the conditions
$$\|T_t x\|_p \leq C\|x\|_p\;,\;\forall \; t\in \real^+,\;\forall \;x\in L_p(\M)$$
and
$$\|T_tx\|_q\leq C\phi(t)^{-\alpha}\|x\|_p\;,\;\forall\; t\in \real^+,\;\forall\; x\in L_p(\M)$$
for fixed $1\leq p < q \leq \8$ and $\alpha=\frac{1}{p}-\frac{1}{q}$. Assume in additional that $p<r<q$ and $\beta=\frac{1}{p}-\frac{1}{r}.$ Then $\psi(L)^{-\beta}$ is of weak type $(p,r)$:
$$\|\psi(L)^{-\beta}x\|_{r,\8}\leq C_{p,q,r}\|x\|_p,\;\;\forall x\in L_p(\M).$$

\end{lem}
\begin{proof}
Observe that the $\Delta_2$ condition of $\phi$ implies that $\phi$ grows at most polynomially at infinity and
$$|\psi(z)^{-\beta}|\leq C|\psi_\beta (z)|\leq C\int_0^\8 e^{-zt}\phi(t)^\beta \frac{dt}{t}\;\;<\8,\;\;\forall z \in \real^+.$$
Thus, if $x$ is in $L_2(\M)\cap L_p(\M)$ and $k$ is large enough, $(I-T_t)^k x$ lies in $\rm{Dom} (\psi(L)^{-\beta})$ for all $t\in\real^+.$ Moreover, as $t$ tends to $\8,$ $(I-T_t)^k x$ tends to $x$ in $L_p(\M)$, since the spectral projection onto the kernel of $L$ is trivial. So that $L_p(\M)\cap {\rm{Dom} }(\psi(L)^{-\beta})$ is dense in $L_p(\M).$

Take an $x\in L_p(\M)\cap {\rm{Dom}} (\psi(L)^{-\beta})$. We need to estimate the $L_{r,\8}$ norm of $\psi(L)^{-\beta}x.$
But by \eqref{mweakqq} and the regular property (6) of $(\phi, \psi)$,
$$\|\psi(L)^{-\beta}x\|_{r,\8}= \|\kappa(L)\circ \psi_{\beta}(L)x\|_{r,\8}\leq C\| \psi_{\beta}(L)x\|_{r,\8},$$
for function $\kappa(z)=[\psi(z)^{\beta}\psi_{\beta}(z)]^{-1}$ whose absolute value is bounded and bounded away from $0$ in any proper sub-cone $\Gamma_\omega \subset \Gamma.$ Thus we are reduced to estimate the $L_{r,\8}$ norm of $\psi_\beta (L)x.$

  By definition, we need to control the trace of $\un_{(\epsilon,\infty)}(|\psi_\beta (L)x |)$, namely, $\lambda_\epsilon (\psi_\beta (L)x )$ for $\epsilon>0.$ To this end, we take $s\in\real^+$ to be chosen later. By the spectral theory,
\be\begin{split}
\psi_\beta (L)x &=\int_0^\8 \phi_\beta (t)T_t x\frac{dt}{t}\\
&=\int_0^s \phi_\beta (t)T_t x\frac{dt}{t}+\int_s^\8 \phi_\beta (t)T_t x\frac{dt}{t}\\
&\triangleq B_{\beta,s}x+D_{\beta,s}x.
\end{split}\ee
We estimate the two parts in the sum. By the property (3) of $\phi$,
\be\begin{split}
\| B_{\beta,s}x\|_p&\leq \int_0^s \phi_\beta (t)\|T_t x\|_p\frac{dt}{t}\\
&\leq C \int_0^s \phi(t)^\beta \| x\|_p \frac{dt}{t} \leq C  \phi(s)^\beta \|x\|_p,
\end{split}\ee
while
\be\begin{split}
\| D_{\beta,s}x\|_q &\leq \int_s^\8 \phi_\beta (t)\|T_t x\|_q \frac{dt}{t}\\
&\leq C \int_s^\8  \phi(t)^{\beta-\alpha} \| x\|_p \frac{dt}{t} \leq C  \phi(s)^{\beta-\alpha} \|x\|_p.
\end{split}\ee
Now since $$\lambda_\epsilon (\psi_\beta (L)x )\leq \lambda_{\frac \epsilon  2}( B_{\beta,s}x )+\lambda_{\frac \epsilon 2} (D_{\beta,s}x )\;,$$
if $s$ is chosen so that $\phi(s)=\big(\frac{\|x\|_p}{ \epsilon}\big)^{\frac{(q-p)}{(\beta p-\beta q + \alpha q)}},$ then we have from the Chebychev inequality that
\be\begin{split}
\lambda_\epsilon (\psi_\beta (L)x )&\leq \big(\frac{\|B_{\beta,s}x \|_p}{\frac 1 2  \epsilon }\big)^p +\big(\frac{\|D_{\beta,s}x \|_q}{\frac 1 2 \epsilon }\big)^q\\
&\leq \big(\frac{2C\phi(s)^\beta \|x \|_p}{\epsilon }\big)^p +\big(\frac{2C \phi(s)^{\beta-\alpha}\| x \|_p}{\epsilon}\big)^q\\
&=C_{p,q,r}\big(\frac{\|x\|_p}{\epsilon} \big)^r,
\end{split}
\ee
which finishes the proof.

\end{proof}

\begin{lem}\label{viiimpliesiii}
Suppose that $(T_t)_{t>0}$ is a symmetric contraction semigroup and that
$$\|\psi(L)^{-\alpha} x\|_{q,\8}\leq C\|x\|_p,\;\;\forall x\in L_p(\M)$$
for some $p,q $ and $\alpha$ satisfying $1\leq p<q\leq \8$ and $\alpha =\frac{1}{p}-\frac{1}{q}<1.$ Then if $p<r<q$ and $\beta =\frac{1}{p}-\frac{1}{r},$
$$\|T_tx\|_r\leq A_{p,q,r} \phi(t)^{-\beta}\|x\|_p,\;\;\forall t\in\real^+,\;\forall x\in L_p(\M).$$
\end{lem}
\begin{proof}
We suppose first that $p>1.$ By the spectral theory, we can write for $x\in L_2(\M)\cap L_p(\M)$ that
$$T_tx =\psi(L)^{-\alpha}\big(\psi(L)^{\alpha} e^{-tL} \big)x.$$
Again, by \eqref{mstrongpp}, we find from the hypothesis of the lemma that
$$\|T_tx\|_{q,\8}\leq C\|\psi(L)^{\alpha} e^{-tL} x\|_p\leq C \sup_{z\in \Gamma_\omega}|\psi(z)^{\alpha} e^{-tz}| \; \|x\|_p,$$
for some $\omega \in (0, \frac{\pi}{2}).$ The hypothesis that the pair $(\phi,\psi)$ is regular implies that
$$\sup_{z\in \Gamma_\omega}|\psi(z)^{\alpha} e^{-tz}|\leq C\sup_{z\in \real^+}|\psi(z)^{\alpha} e^{-tz}|.$$
Thus, if we can establish that
\beq\label{phicontrolpsi}
|\psi(z)^{\alpha} e^{-tz}|\leq C \phi(t)^{-\alpha},\;\;\forall t\in \real^+,\;\forall z\in \real^+,
\eeq
then it will follow that
$$\|T_tx\|_{q,\8}\leq C \phi(t)^{-\alpha}\|x\|_p,\;\;\forall t\in\real^+,\;\forall x\in L_p(\M).$$
Interpolating this with the estimate $\|T_tx\|_p\leq \|x\|_p$ by the real method will then give us that for $p\leq r<q$ and $\beta =\frac{1}{p}-\frac{1}{r}$,
$$\|T_tx\|_r\leq A_{p,q,r}\phi(t)^{-\beta}\|x\|_p,\;\;\forall t\in \real^+,\;\forall x\in L_p(\M),$$
as desired.

The claim \eqref{phicontrolpsi} has already been established in \cite{CM1993}. Recall that $\psi_{\alpha}(z)=\int_0^\8 e^{-zs} \phi_{\alpha}(s)\frac{ds}{s}$ for all $z\in\real^+.$ Accordingly, for all $t,z \in\real^+$, we have
\be\begin{split}
|\psi(z)|^{-\alpha}&\geq C \int_0^\8 e^{-zs} \phi^{\alpha}(s)\frac{ds}{s}\geq C \int_{\frac t 2}^t e^{-zs} \phi^{\alpha}(s)\frac{ds}{s}\\
&\geq C e^{-tz}\phi(\frac t 2)^\alpha \int_{\frac t 2}^t \frac{ds}{s}\geq  C e^{-tz}\phi(t)^\alpha ,
\end{split}
\ee
which ensure \eqref{phicontrolpsi}.

If $p=1,$ then $1<q<\8.$ We decompose $T_tx =\big(\psi(L)^{\alpha} e^{-tL} \big)\psi(L)^{-\alpha}x$, and observe that $\psi(L)^{-\alpha}x \in L_{q,\8}(\M).$ In view of \eqref{mweakqq}, a similar argument as the case $p>1$ gives the proof of the case $p=1.$

\end{proof}

Now we can conclude the proof of Theorem \ref{phipsiultra}. By using the real interpolation method between (iv) and its dual estimate, namely
$$\|\psi(L)^{-\alpha}f\|_\8\leq C_r\|f\|_{r',1}\;\;\;\forall f\in L_{r',1}(\M),$$
one can show that (iv) implies (v), and that the constant in (v) depends linearly on that in (iv). Trivially, (v) implies (vi) and (vi) implies (vii). These complete the proof of the theorem.

\section{Multiplier and Maximal Operators}

\subsection{Multiplier operators}
In this section, $L$ is again the generator of a $\phi$-ultracontractive semigroup $(T_t)_{t>0}$ on $\M,$ and $(\phi,\psi)$ a pair of regularly related functions. We intend to find conditions on the function $m$ such that the operator $m(L)$ extends to a bounded operator from $L_p(\M)$ to $L_q(\M).$ Firstly we treat the case $p=q.$ For a Borel measurable function $m(\cdot)$ on $\real^+$, its Mellin $\xi$-transform is defined by
 $$[\mathfrak{M}_\xi m ](u)=\int_0^\8 \eta ^{\xi-iu} m(\eta) \frac{d\eta}{\eta},\;\;\forall \,u\in\real.$$
 Given a nonnegative integer $N$, we denote $m_N(t,\eta)$ the function
$$m_N(t,\eta)=(t\eta)^N e^{-\frac{t\eta}{2}}m(\eta),\;\;\forall t,\eta \in \real^+,$$
and $\mathfrak{M}_\xi m_N$ the Mellin $\xi$-transform of $m_N$ with respect to the second variable is then
$$[\mathfrak{M}_\xi m_N ](t,u)=\int_0^\8 \eta ^{\xi-iu} m_N(t,\eta) \frac{d\eta}{\eta},\;\;\forall t\in\real^+,\;\forall u\in \real.$$
In particular, $\mathfrak{M} m_N$ denotes the Mellin transform of $m_N$:
$$[\mathfrak{M}  m_N ](t,u)=\int_0^\8 (t\eta)^N  e^{-\frac{t\eta}{2}}m(\eta) \eta ^{-iu} \frac{d\eta}{\eta},\;\;\forall t\in\real^+,\;\forall u\in \real.$$
If $\mathfrak{M}_\xi  m$ is integrable with respect to the Lebesgue measure on $\real$, the inverse Mellin transform holds:
$$m(\eta)=\frac{1}{2\pi}\int_{-\8}^\8 \eta ^{iu-\xi} [\mathfrak{M}_\xi m](u) du,\;\;\forall \eta\in \real^+.$$

\begin{thm}\label{multipliermain}
Assume that $1<p<\8$. Suppose that $m$ is a Borel measurable function on $\real^+$ such that for some nonnegative integer $N$,
\beq\label{hypop}
\int_{-\8}^\8 \sup_{t>0}|[\mathfrak{M}m_N](t,u)|\;\; |||L^{iu}|||_pdu < \8.
\eeq
Then $m(L)$ is bounded on $L_p(\M)$. Similarly, if
\beq\label{hypoweakp}
\int_{-\8}^\8 \sup_{t>0}|[\mathfrak{M}m_N](t,u)|\;\; |||L^{iu}|||_{p,\8}du < \8.
\eeq
then $m(L)$ is bounded on $L_{p,\8}(\M).$
\end{thm}

\begin{proof}
Denote functions $c:\real \rightarrow \real^+$ and $a_u: \real^+ \rightarrow\real $ by
$$c(u)=\sup_{t>0}|[\mathfrak{M}m_N](t,u)|,\;\;a_u(t)=\frac{[\mathfrak{M}m_N](t,u) }{c(u)}.$$
Obviously, $\|a_u\|_\8\leq 1.$ By the inverse Mellin transform formula,
\be
\begin{split}
m(\eta)&= \frac{1}{ \Gamma(N+1)}\int_0^\8 (t\eta)^{N+1} e^{-t\eta}m(\eta) \frac{dt}{t}\\
&=\frac{1}{ \Gamma(N+1)}\int_0^\8 (t\eta) e^{\frac{-t\eta}{2}}m_N(t,\eta) \frac{dt}{t}  \\
&=\frac{1}{2\pi \Gamma(N+1)} \int_{-\8}^{\8} \int_0^\8 a_u(t)\, t\eta\, e^{-\frac{t\eta}{2}} \frac{dt}{t}\; c(u)\eta^{iu}\,du .
\end{split}\ee
Denote the function $\Lambda_u: \eta \mapsto \int_0^\8 a_u(t)t\eta \,e^{-\frac{t\eta}{2}}\,\frac{dt}{t}$. Then we have formally
\beq\label{decompositionmL}
m(L)=\frac{1}{2\pi \Gamma(N+1)}\int_{-\8}^\8 \Lambda_u(L) c(u)L^{iu} du. \eeq
Now we need to check that this decomposition of the operator $m(L)$ is legitimate.

Indeed, the function $\Lambda_u$ defined above extends to a holomorphic function in the right half-plane of $\mathbb{C}$, and for $\eta \in \Gamma_{\frac \pi 2-\varepsilon}$ with $\e>0$,
\be\begin{split}
|\Lambda_u(\eta)|&=\bigg|\int_0^\8 a_u(t)t\eta \,e^{-\frac{t\eta}{2}}\;\frac{dt}{t} \bigg|\\
&\leq \int_0^\8 t|\eta|\, e^{-\frac{t|\eta| \sin \varepsilon}{2}}\,\frac{dt}{t} =\frac{2}{\sin \varepsilon}.
\end{split}
\ee
Thus, by \eqref{mstrongpp}, $|||\Lambda_u(L)|||_p\leq C_p$ for all $u$ in $\real.$ Therefore, \eqref{hypop} ensures that
\beq\label{controlmL}
\int_{-\8}^\8 |||\Lambda_u(L)|||_p\; |c(u)| \;|||L^{iu}|||_p du < \8.
\eeq
On the other hand, since $u\mapsto a_u(\eta)$ and $u\mapsto \eta^{iu}$ are continuous functions for each $\eta\in\real^+,$ $u\mapsto a_u(L)$ and $u\mapsto L^{iu}$ are strong operator topology continuous functions. Finally, we find that decomposition \eqref{decompositionmL} is convergent, and finish the proof of the strong type conclusion. Repeating the argument above, we can prove the weak type conclusion.

\end{proof}

\medskip

Some consequences of the above theorem are immediate. In the following one, the hypothesis of Theorem \ref{multipliermain} can be verified by the classical Paley-Wiener Theorem for Mellin transforms.
\begin{cor}\label{multiplierholo}
Suppose that $1<p<\8$ and $A, \e >0,$ and that $|||L^{iu}|||_p\leq A e^{\e |u|}$ for all $u\in \real.$ If $m$ extends to a bounded holomorphic function $m_{\omega}$ on $\Gamma_\omega,$ where $\omega > \e,$ then $m(L)$ extends to a bounded operator on $L_p(\M)$, and $|||m(L)|||_p\leq C_{\omega, \e} A \|m_\e \|_{\8}$.
\end{cor}

\begin{proof}
Since $\omega > \e>0$, we can choose $\e_1\in (\e,\omega)$. By the relationship between Mellin transforms and Fourier transforms, we have
$$[\mathfrak{M}  m_N ](t,u)=\int_0^\8\eta^{-iu} m_N(t,\eta)\frac{d\eta}{\eta} = \widehat{F_t}(u) ,$$
where $F_t(\cdot)=m_N(t,e^{\cdot})$. Moreover,
\be\begin{split}
[\mathfrak{M}  m_N ](t,u)e^{\e_1 u}&=\int_0^\8\eta^{-iu}e^{\e_1 u} m_N(t,\eta)\frac{d\eta}{\eta} \\
&=\int_0^\8\eta^{-iu} m_N(t,\eta e^{-i\e_1})\frac{d\eta}{\eta}
= \widehat{G_t}(u) ,
\end{split}\ee
with $G_t(\cdot)=m_N(t,e^{\cdot-i\e_1})$. Since $m$ extends to a bounded holomorphic function $m_{\omega}$ on $\Gamma_\omega$ with $\omega>\e_1$, we can apply classical Paley-Wiener Theorem to $G_t(\cdot)$. This gives that $[\mathfrak{M}  m_N ](t,u)e^{\e_1 u}$ is $L_2(\real^+,du)$ bounded, uniformly on $t>0$. By a similar argument for $m_N(t,e^{\cdot+i\e_1})$, we see that $[\mathfrak{M}  m_N ](t,u)e^{-\e_1 u}$ is $L_2(\real\setminus \real^+,du)$ bounded, uniformly on $t>0$. Thus \eqref{hypop} is fulfilled by the above discussion and the H\"{o}lder inequality, since $|||L^{iu}|||_p\leq A e^{\e |u|}$.

\end{proof}

\medskip

Another special case of Theorem \ref{multipliermain} concerns the function satisfying the H\"{o}rmander condition. Suppose that $\chi\in \mathbb{N}$ and $\alpha \in\real^+.$ A function $m$ satisfies the H\"{o}rmander condition of order $(\alpha, \chi)$ if there exists a constant $C$ such that
$$\sup_{R>0} |\psi(R)|^{\alpha}\int_R^{2R} |\eta ^j m^{(j)}(\eta)|\frac{d\eta}{\eta} \leq C,\;\;\forall j\in \{0,1,\cdots,\chi\}.$$
The smallest constant $C$ for which the above inequality holds is called the H\"{o}rmander $(\alpha, \chi)$-constant of $m$. For such a function $m$, one can check that, by a direct calculation, the hypothesis of Theorem \ref{multipliermain} holds (see \cite[Theorem 4]{Meda1990}). Thus we have the following

\begin{cor}\label{hormanderpp}
Suppose that $m$ satisfies the H\"{o}rmander condition of order $(0, \chi)$ with H\"{o}rmander $(0, \chi)$-constant $C$. Fix $p\in(1,\8)$ and $A,\xi \in\real^+.$ If $|||L^{iu}|||_p\leq A(1+|u|)^\xi$ for all $u\in \real$ and $\chi>\xi+1,$ then $m(L)$ extends to a bounded operator on $L_p(\M)$, and $|||m(L)|||_p \leq C_\xi AC.$ Similarly, if $|||L^{iu}|||_{p,\8} \leq A(1+|u|)^\xi$ for all $u\in \real$ and $\chi>\xi+1,$ then $m(L)$ extends to a bounded operator on $L_{p,\8}(\M)$, and $|||m(L)|||_{p,\8} \leq C'_\xi AC.$

\end{cor}

In the rest of this subsection, we prove some results for the case $p<q.$ We assume that there exists a positive number $\sigma$ such that for $1<p<\8,$
\beq\label{Liup}
|||L^{iu}|||_p\leq C_p (1+|u|)^{\sigma |\frac 1 p-\frac 1 2|},\;\;\forall u\in \real\eeq
or
\beq\label{Liupweak}
|||L^{iu}|||_{p,\8}\leq C'_p (1+|u|)^{\sigma |\frac 1 p-\frac 1 2|},\;\;\forall u\in \real.\eeq
A notation is introduced here for convenience: Given $p,q$ satisfying $1\leq p\leq q \leq \8,$ $d([p,q],2)$ is defined to be $\min_{r\in [p,q]}|\frac{1}{r}-\frac{1}{2}|.$

\begin{cor}
Assume that $m$ extends to an analytic function in the cone $\Gamma _\omega$ and that $$\sup_{\eta\in \Gamma_\omega}|\psi(\eta)^\alpha m(\eta)|<\8$$ for some $\alpha \in (0,1).$ Then $m(L)$ extends to a bounded operator from $L_p(\M)$ to $L_q(\M)$, provided that $1\leq p< q \leq \8$, $d([p,q],2)<\frac{\omega}{\pi}$ and $\alpha =\frac{1}{p}-\frac{1}{q},$ and to an operator of weak type $(1,r)$ if $1<r<\8,$ $d([1,r],2)<\frac{\omega}{\pi}$ and $\alpha =1-\frac{1}{r}.$
\end{cor}
\begin{proof}
We consider only the strong type result, since the weak one will follow from the same argument. Choose $r\in [p,q] $ such that $|\frac{1}{r}-\frac{1}{2}|= d([p,q],2),$ and denote
$$\beta=\frac{1}{p}-\frac{1}{r},\;\; \gamma= \frac{1}{r}-\frac{1}{q}.$$
Then for all $\eta\in \real^+,$ $m(\eta)=\psi(\eta)^{-\gamma}[\psi(\eta)^\alpha m(\eta)]\psi(\eta)^{-\beta},$ which gives
$$m(L)=\psi(L)^{-\gamma}[\psi(L)^\alpha m(L)]\psi(L)^{-\beta}.$$
By the hypothesis of $\phi$-ultracontractivity of the semigroup $(T_t)_{t>0},$ we find that the first factor $\psi(L)^{-\gamma}$ is bounded from $L_p(\M)$ to $L_r(\M)$, and the third factor $\psi(L)^{-\beta}$ bounded from $L_r(\M)$ to $L_q(\M).$ By Corollary \ref{multiplierholo}, and the hypothesis on the function $m$, the second one $[\psi(L)^\alpha m(L)]$ is bounded on $L_r(\M)$. In summary, we have proved the desired conclusion.
\end{proof}

A direct consequence from the above conclusion is for those $m$ that satisfy the H\"{o}rmander condition, by the aide of Corollary \ref{hormanderpp}.

\begin{cor}
Suppose that $m$ satisfies a H\"{o}rmander condition of order $(\alpha,\chi)$ where $0<\alpha <1$ and $\chi\in \mathbb{N}$ less than $\frac{\sigma}{2} +1$. Then $m(L)$ extends to a bounded operator from $L_p(\M) $ to $L_q(\M)$ for $1<p<q<\8$ and $\alpha =\frac{1}{p}-\frac{1}{q}$, and to an operator of weak type $(1,r)$ for $1<r<\8$ and $\alpha=1-\frac{1}{r}.$
\end{cor}
\begin{remark}
The above hypothesis of $m$ can be replaced by a stronger condition:
$$\sup_{\eta \in\real^+} |\psi(\eta)^{\alpha}\eta ^j m^{(j)}(\eta)|\leq C,\;\;\forall j\in\{0,1,\cdots,\chi\},$$
for some integer no less than $\frac{\sigma}{2}+1$ and some $\alpha \in (0,1).$

\end{remark}

\subsection{Maximal operators}
To present the results on noncommutative maximal operators, we need to recall some knowledge of noncommutative maximal functions. These were first introduced by Pisier \cite{Pisier1998} and Junge \cite{Junge2002}. Let $1\leq p\leq \8,$ $L_p(\M;\,\ell_\8)$ is defined to be the space of all sequences $x=(x_n)_{n\geq 1}$ in $L_p(\M),$ which admit a factorization of the following form: there exist $a,b \in L_{2p}(\M)$ and a bounded sequence $y=(y_n)$ in $L_\8(\M)$ such that
$$x_n =ay_n b , \;\; n\geq 1.$$
The norm of $x\in L_p(\M;\,\ell_\8)$ is given by
$$\|x\|_{L_p(\M;\,\ell_\8)} =\inf \{\|a\|_{2p}\sup_{n\geq 1}\|y_n\|_\8 \|b\|_{2p}\},$$
where the infimum is taken over all such factorizations above. This norm makes $L_p(\M;\,\ell_\8)$ a Banach space. The norm is usually denoted by $\|{\sup}^+_{n\geq 1} x_n\|_p$. But note that $\|{\sup}^+_{n\geq 1} x_n\|_p$ is just a notation since $\sup_{n\geq 1} x_n$ does not make any sense in the noncommutative setting. Moreover, this definition can be extended to an arbitrary index set $J$: the space $L_p(\M;\,\ell_\8(J)).$ One can easily check that a family $(x_j)_{j\in J}$ in $L_p(\M)$ belongs to $L_p(\M;\,\ell_\8(J))$ if and only if
$$\sup_{J_1 \subset J, J_1\,\mbox{ finite }} \|{\sup_{j\in J_1} }^+ x_j\|_p \leq \8.$$
And if this is the case, we have
$$ \|{\sup_{j\in J}}^+ x_j\|_p =\sup_{J_1 \subset J, J_1\,\mbox{ finite}} \|{\sup_{j\in J_1}} ^+x_j\|_p.$$
More recently, Dirksen considered more general maximal operator spaces in the noncommutative setting, which allows us to consider the space $L_{p,q}(\M;\,\ell_\8),$ the space of sequences $x=(x_n)_{n\geq 1}$ with entries in noncommutative Lorentz space $L_{p,q}(\M).$ We omit the details and refer the reader to \cite{Dirksen2013}.

In this subsection we focus on the maximal operator
\beq\label{maximalop}
x \mapsto (t^{\alpha} m(tL)x)_{t>0},\eeq
where $\alpha$ is a nonnegative real number.
The following theorem gives some sufficient conditions on $m$ so that the above maximal operator is of strong (or weak) type $(p,q)$, i.e. bounded from $L_{p}(\M)$ to $L_{q}(\M;\,L_\8(\real^+))$ (or $L_{q,\8}(\M;\,L_\8(\real^+))$ respectively).  Here, we deal with the semigroups $(T_t)_{t>0}$ satisfying the estimate that for some positive $\nu$,
$$\|T_t x\|_\8 \leq A t^{-\nu} \|x\|_1,\;\;\forall t\in \real^+,\;\forall x\in L_1(\M),$$
which we call $\nu$-ultracontractivity.
\begin{thm}
Let $(T_t)_{t>0}$ be a $\nu$-ultracontractive symmetric contraction semigroup with the generator $L$, so that $L^{-\alpha}$ (by Theorem \ref{phipsiultra}) is of weak type $(1,r)$ and of strong type $(p,q),$ where $1<p<q<\8, 1<r<\8 $, and $1-\frac{1}{r} =\frac{1}{p}-\frac{1}{q} =\frac{\alpha}{\nu}$. If $m$ is a Borel measurable function on $\real^+$ whose Mellin $\alpha$-transform $\mathfrak{M}_{\alpha}m$ satisfies the estimate
$$\int_{-\8}^{\8} |[\mathfrak{M}_{\alpha}m](u)|\;|||L^{iu-\alpha}|||_{p\ra q}du <\8,$$
then the maximal operator \eqref{maximalop} is of strong type $(p,q).$ Similarly, if
$$\int_{-\8}^{\8} |[\mathfrak{M}_{\alpha}m](u)|\;|||L^{iu-\alpha}|||_{1\ra (r,\8)}du <\8,$$
then the operator \eqref{maximalop} is of weak type $(1,r).$
\end{thm}
\begin{proof}
As usual, we just prove the strong type conclusion of this theorem.
By the spectral theory and the inverse Mellin transform, we have
$$m(tL)x = (2\pi)^{-1} \int_{-\8}^{\8} [\mathfrak{M}_{\alpha}m](u) (tL)^{iu-\alpha} x \,du, $$
whence
$$t^\alpha m(tL)x = (2\pi)^{-1} \int_{-\8}^{\8} [\mathfrak{M}_{\alpha}m](u) t^{iu} L^{iu-\alpha} x\, du\,. $$
Taking the $L_q(\M;\,L_\8(\real^+))$-norm and applying the triangle inequality of this norm, we get
$$ \|{\sup_{t>0}} ^+\, t^\alpha m(tL)x \|_q \leq (2\pi)^{-1} \int_{-\8}^{\8} [\mathfrak{M}_{\alpha}m](u) \, \|{\sup_{t>0}}^+\, t^{iu} L^{iu-\alpha} x\|_q \;du$$
Now by the hypothesis on $m$ and $L$, we obtain
\be\begin{split}
\|{\sup_{t>0}}^+ \, t^\alpha m(tL)x \|_q& \leq (2\pi)^{-1} \int_{-\8}^{\8} [\mathfrak{M}_{\alpha}m](u)\, \| L^{iu-\alpha} x\|_q \;du\\
&\leq  (2\pi)^{-1} \int_{-\8}^{\8} [\mathfrak{M}_{\alpha}m](u) \;||| L^{iu-\alpha} |||_{p\ra q}\; du\;\|x\|_p =C\|x\|_p.
\end{split}
\ee

\end{proof}

\section{Ultracontractivity and Logarithmic Sobolov Inequalities}\label{log-Sb}

In this section we consider the $\phi$-ultracontractive semigroup with $\phi(t)=e^{-M(t)}$ for some decreasing function $M(t)$. Namely,
\beq\label{M-u.c.}
|||T_t|||_{1\ra \8}\leq  e^{M(t)},\;\; t>0 .
\eeq
This $\phi$ may not satisfy the conditions in the definition of regularly related in general. So we can not get Sobolev inequality for its generator from Theorem \ref{phipsiultra}. But we can follow the method in \cite{Davies}, to characterize \eqref{M-u.c.} by the following logarithmic Sobolov inequality:
\beq\label{log-Sobolev}
\tau(x^2 \log x) \leq \e \tau(xLx) +C M(\e)\|x\|_2^2+\|x\|_2^2\log \|x\|_2\,.
\eeq
Here again $C$ is a positive constant. The main result of this section is

\begin{thm}\label{u.c.&logS}
Let $(T_t)_{t>0}$ be a ultracontractive semigroup satisfying \eqref{M-u.c.}. Then there exist a positive constant $C$ such that \eqref{log-Sobolev} holds for all positive $x\in {\rm{Dom}}(L)\cap L_1(\M)\cap L_\8(\M)$ and all $\e>0$.

Conversely, if \eqref{log-Sobolev} is true for all positive $x\in {\rm{Dom}}(L)\cap L_1(\M)\cap L_\8(\M)$ and all $\e>0$, then \eqref{M-u.c.} holds for the function $\overline M (t) =\frac {2C} t \int _0^t  M(s)ds$.

\end{thm}

Our argument is standard, following the classical proofs in \cite{Davies}. See also \cite{OZ1999} for the noncommutative case. For concrete examples about the function $M(t)$, we refer to \cite[Section~2.3]{Davies}. Before passing to the proof, we do some preparation. The following two lemmas are well-known. See \cite{OZ1999}.

\begin{lem}\label{diff-Lp}
Let $x\in L_1(\M)\cap L_\8(\M)$ be positive. Then $q\mapsto \|x\|_q^q$ is differentiable on $q\in (1,+\8)$ and
$$\frac{d}{dq} \|x\|_q^q =\tau (x^q \log x).$$
\end{lem}

\begin{lem}\label{diff-logLp}
Let $x\in {\rm{Dom}}(L) \cap L_1(\M)\cap L_\8(\M)$ be a strictly positive element and let $q:\real_+ \ra (1,+\8)$ be a $C^1$ function. Then $t\mapsto \tau \big((T_tx)^{q(t)}\big)$ and $t\mapsto \log \| T_t x\|_{q(t)}$ are both differentiable:
\be
\frac{d}{dt}\tau \big((T_tx)^{q(t)}\big) =-q(t)\tau \big((T_tx)^{q(t)-1}L(T_tx) \big) +q'(t) \tau \big((T_tx)^{q(t)} \log (T_tx)\big)
\ee
and
\be\begin{split}
\frac{d}{dt} \log \|T_t x\|_{q(t)} = & \frac{q'(t)}{q(t)\|T_tx\|_{q(t)}^{q(t)}} \big[ \tau \big( (T_tx)^{q(t)} \log (T_tx)   \big) - \|T_tx\|_{q(t)}^{q(t)} \log \|T_tx\|_{q(t)}\big]\\
& -\frac{1}{\|T_tx\|_{q(t)}^{q(t)}} \tau \big( (T_tx)^{q(t)-1} L (T_tx)   \big).
\end{split}
\ee
\end{lem}

We also need the \emph{$L_p$ regularity} of Dirichlet forms ${\rm{Dom}}(L)\ni x\mapsto \tau(x^\alpha Lx)$, $\alpha>1$. Note that $(T_t)_{t>0}$ in our paper is positive, symmetric, contractive. Thus the Dirichlet forms process automatically the $L_p$ regularity. The following proposition was deduced by a group of Professor Quanhua Xu, in a seminar on the topic of hypercontractivity, in Wuhan University 2013. See also \cite{RX2015}.
\begin{prop}
\label{Lp regularity}
Let $x\in  {\rm{Dom}}(L)$ be a strictly positive element and $q\geq 2$, $\frac{1}{q}+\frac{1}{q^*}=1$. Then we have
$$\tau\left( x^{\frac{2}{q^*}}L(x^{\frac{2}{q}})\right)\leq \frac{q^2}{4(q-1)}\tau(xLx) .$$
\end{prop}
The proof requires the following lemma.
\begin{lem}
Suppose that $\mathcal{A}$ is a ultraweakly dense C*-subalgebra of $\M$. Let $x\in  \mathcal A $ be a normal element and $P: \mathcal A  \ra \mathcal A $ be a positive map. Denote $K=\sigma(x)$. Then there exists a positive measure $\mu$ on $K\times K$ with $\| \mu\|\leq \| P\|$ such that
$$\forall\phi,\psi\in C(K),\quad \tau\left( \phi(x)P(\psi(x))\right) =\int_{K\times K}\phi(t)\psi(s)\,d\mu(t,s).$$
\end{lem}
\begin{proof}
Set $\Phi:C(K)\to \mathcal A $ to be the function calculus of $x$. Then $\Phi$ and $P\circ \Phi$ are positive, and hence are completely positive since $C(K)$ is a commutative algebra (see e.g. \cite[34.4]{conway2000operator}). Thus $\Phi\otimes (P\circ \Phi)$ extends to a completely positive map from $C(K)\otimes_{\max}C(K)$ to $\mathcal A \otimes_{\max} \mathcal A $ (\cite[Corollary 11.3]{pisier2003operator}). Define a state $W:\mathcal A \otimes_{\max} \mathcal A \to \com$ by $W(a\otimes b)=\tau(ab)$. Then $W\circ(\Phi\otimes (P\circ\Phi))$ is a positive map on $C(K)\otimes_{\max}C(K)$ which is equal to $C(K)\otimes_{\min}C(K)=C(K\times K)$ since $C(K)$ is nuclear (\cite[6.4.15]{murphy1990Cstar}). Therefore by Riesz's theorem, there exists a positive measure $\mu$ on $K\times K$ with $\| \mu\|\leq \|W\circ(\Phi\otimes (P\circ\Phi))\| \leq \|\Phi\|\|P\circ\Phi\| \leq \| P\|$ such that for any $f\in C(K\times K)$,
$$W\circ(\Phi\otimes (P\circ\Phi))(f)=\int_{K\times K}f\,d\mu,$$
which is exactly the desired result if we take $f=\phi\otimes\psi$.
\end{proof}
\begin{proof}[Proof of Proposition \ref{Lp regularity}]
Note that the domain ${\rm{Dom}}(L)$ consists of all elements $x$ such that
$$\lim_{t\ra 0}\frac{T_tx-x}{t}$$
exists, so $x\in {\rm{Dom}}(L)$ implies $x^{\frac{2}{q}}\in {\rm{Dom}}(L)$ if $x$ is invertible. Thus $\tau ( x^{\frac{2}{p}}L(x^{\frac{2}{q}}) )$ is well-defined in the inequality.

Now we fix $t>0$ and apply the precedent lemma with $P=T_t$. We consider the function
$$h:\real\ra \real,\quad \theta\mapsto \tau(x^{2-\theta}T_t(x^\theta)).$$
The lemma yields that for $K=\sigma(x)\subset (0,+\8)$,
$$h(\theta)=\int_{K\times K}u^{2-\theta}v^\theta\,d\mu(u,v)=\int_{K\times K}u^{2}\big(\frac{v}{u}\big)^\theta\,d\mu(u,v)$$
whence
$$h^{(k)}(\theta)=\int_{K\times K}u^2\left( \frac{v}{u}\right) ^\theta\left( \log\left( \frac{v}{u}\right) \right) ^k\,d\mu(u,v)\,.$$
Therefore $h^{(2k)}(\theta)\geq 0$ for all $k\geq 1$. Moreover for $\theta\in [0,1]$, by symmetry,
$$h(1+\theta)=\tau(x^{1-\theta}P_t(x^{1+\theta}))=\tau(P_t(x^{1-\theta})x^{1+\theta})=h(1-\theta),$$
so $h^{(2k+1)}(1)=0$ for all $k\geq 0$.

Therefore by Taylor's formula, for all $\theta\in [0,2]$,
\begin{align*}
h(\theta)&=h(1)+\sum_{k=1}^{+\8}\frac{h^{(2k)}(1)}{(2k)!}(\theta-1)^{2k}\leq h(1)+(\theta-1)^2\sum_{k=1}^{+\8}\frac{h^{(2k)}(1)}{(2k)!}\\
&=h(1)+(\theta-1)^2(h(2)-h(1)).
\end{align*}
Take $\theta=\frac{2}{q}$ in the above inequality. We have $h(\frac{2}{q})=\tau ( x^{\frac{2}{q^*}}L(x^{\frac{2}{q}}) )$, $h(1)=\tau(xT_tx)$, $h(2)=\tau(T_t(x^2))\leq \tau(x^2)$ and
\begin{align*}
\tau\left( x^{\frac{2}{q^*}}(T_t(x^{\frac{2}{q}})-x^{\frac{2}{q}})\right)
&=\tau\left( x^{\frac{2}{q^*}} T_t(x^{\frac{2}{q}})\right) -\tau(x^2)=h(\frac{2}{q})-\tau(x^2)\\
&\leq h(1)+(\frac{2}{q}-1)^2(h(2)-h(1))-\tau(x^2)\\
&\leq \tau(xT_tx)+(\frac{2}{q}-1)^2(\tau(x^2)-\tau(xT_tx))-\tau(x^2)\\
&=\frac{q^2}{4(q-1)}\tau(x(T_tx-x)).
\end{align*}
Divide the above inequality by $t$ and let $t\ra 0$, then we obtain the desired result.
\end{proof}

Now we are ready to give
\begin{proof}[Proof of Theorem \ref{u.c.&logS}]
We first assume the ultracontractivity of the semigroup. For every real number $y$, denote
$$T_{iy}x=e^{iyL}x=\int_0^\8 e^{iy\lambda}\; dP_\lambda x\;.$$
Then $T_{iy}$ is well-defined on $L_2(\M)$. And since $|e^{iy\lambda}|=1$ for all real $y$ and $\lambda$, $T_{iy}$ is unitary and thus has operator norm $1$ on $L_2(\M)$. If $|||T_t|||_{1\ra \8}\leq  e^{M(t)}$ for all $t>0, $ Theorem \ref{phipsiultra} gives $|||T_t|||_{2\ra \8}\leq A e^{\frac1 2 M(t)}$. Whence $|||T_{t+iy}|||_{2\ra \8}\leq A e^{\frac1 2 M(t)}$ for every $y\in\real$. Now we apply Stein's complex interpolation method \cite{Stein1956} to get
$$|||T_s|||_{2\ra p(s)} \leq A' e^{\frac{sM(t)}{2t}},$$
where $0\leq s <t$ and $p(s)=\frac{2t}{t-s}$. Assume that $\|x\|_2=1$. Then $\|T_sx\|_{ p(s)} \leq A' e^{\frac{sM(t)}{2t}}.$ Therefore,
$$\frac{d}{ds} \|T_sx\|_{ p(s)}^{p(s)} \bigg|_{s=0} \leq A'\frac{M(t)}{t}.$$
By Lemma \ref{diff-logLp}, we have
$$-t\tau (x\,Lx) + \tau (x^2 \log x) \leq\frac{A'}{2}M(t),$$
which completes the proof of this part.

Now we turn to proving the opposite. Assume \eqref{log-Sobolev}. Replacing $x$ by $x^{\frac p 2}$, by the $L_p$-regularity of the Dirichlet form, we have
\beq\label{log-Sobolev-p}
\tau(x^p \log x) \leq \frac{\e p}{2(p-1)} \tau(x^{p-1}Lx) +\frac{2C}{p}M(\e)\|x\|_p^p+\|x\|_p^p\log \|x\|_p\,.
\eeq
Fix $t>0$. Define two functions on $[0,t)$ by
$$p(s) = \frac{2t}{t-s} \,,\;\;N(s) = \frac C t \int_0^s M(t-\sigma)d\sigma.$$
Then
$$p(0)=2,\; N(0)=0 \;\;\mbox{and}\;\;\lim_{s\ra t^-}p(s)=+\8,\;\lim_{s\ra t^-}N(s)=\frac C t \int_0^t M(t-\sigma)d\sigma\,,$$
and
$$p'(s)=\frac{2t}{(t-s)^2}\, ,\;\;N'(s)=\frac C t M(t-s) .$$
Consider the function $F(s) = \log \big(e^{-N(s)}\|T_s x\|_{p(s)}\big).$ Noting that $\frac{p(s)}{p'(s)}\geq \frac{(t-s)p(s)}{2(p(s)-1)},$ we have
\be\begin{split}
\frac{d}{ds} F(s) &= -N'(s) +\frac{p'(s)}{p(s) \|T_sx\|_{p(s)}^{p(s)}}\big[\tau \big( (T_sx)^{p(s)} \log (T_s x)   \big) - \|T_s x\|_{p(s)}^{p(s)} \log \|T_s x\|_{p(s)}\big]\\
&\;\;\;\;-\frac{1}{\|T_sx\|_{p(s)}^{p(s)}} \tau \big( (T_sx)^{p(s)-1} L (T_sx)   \big)\\
&= \frac{p'(s)}{p(s) \|T_sx\|_{p(s)}^{p(s)}}\bigg(-\frac{2}{p(s)}CM(t-s)\|T_sx\|_{p(s)}^{p(s)}+ \tau \big( (T_sx)^{p(s)} \log (T_s x)   \big)\\
&\;\;\;\; - \|T_s x\|_{p(s)}^{p(s)} \log \|T_s x\|_{p(s)}-\frac{p(s)}{p'(s)} \tau \big( (T_sx)^{p(s)-1} L (T_sx)   \big)\bigg)\\
&\leq \frac{p'(s)}{p(s) \|T_sx\|_{p(s)}^{p(s)}}\bigg(-\frac{2}{p(s)}CM(t-s)\|T_sx\|_{p(s)}^{p(s)}+ \tau \big( (T_sx)^{p(s)} \log (T_s x)   \big)\\
&\;\;\;\; - \|T_s x\|_{p(s)}^{p(s)} \log \|T_s x\|_{p(s)}-\frac{(t-s)p(s)}{2(p(s)-1)} \tau \big( (T_sx)^{p(s)-1} L (T_sx)   \big)\bigg).
\end{split}\ee
Putting $\e=t-s$ in \eqref{log-Sobolev-p}, then we have $\frac{d}{ds} F(s)\leq 0$. This yields
$$e^{-N(s)}\|T_s x\|_{p(s)} \leq \|x\|_2\;\;\forall\; s\in [0,t).$$
Whence
$$\|T_t x\|_{p(s)} \leq e^{N(s)}\|x\|_2\;\;\forall\; s\in [0,t).$$
Letting $s\ra t^-$, we obtain
$$\|T_t x\|_{\8} \leq e^{\frac C t \int_0^t M(t-\sigma)d\sigma}\|x\|_2\;.$$
The desired conclusion follows then from a dual argument.
\end{proof}

\section{Localizations}

In this section we study semigroups with local $\nu$-ultracontractive properties:
\beq\label{local-uc}
\|T_tx\|_\8\leq A t^{-\nu}\|x\|_1,\;\forall t\in(0,1),\;\forall x\in L_1(\M).\eeq
Like what Theorem \ref{phipsiultra} has done, we are going to characterize \eqref{local-uc} by (nonhomogeneous) Sobolev inequalities. For the sake of simplicity, we consider only the Sobolev inequalities for $L^{\alpha \nu}$, from $L_2$ to $L_q$, with $\alpha =\frac 1 2 -\frac 1 q$. Our main result of this section is

\begin{thm}\label{local-thm}
Property \eqref{local-uc} is equivalent to \beq\label{non-hom-Sobolev}
\|x\|_q\leq C(\|L^{\alpha \nu}x\|_2+ \|x\|_2)\;,\;\;\forall\; x\in \rm {Dom} (L),
\eeq
where $q>2$, $\alpha=\frac 1 2 -\frac 1 q$.
\end{thm}

Firstly we establish that \eqref{local-uc} implies \eqref{non-hom-Sobolev}. In fact, we can prove this implication for more general case.

\begin{prop}
Let $(\phi,\psi)$ be a pair of regularly relative functions. Assume in additional that
\beq\label{condition-phi}
\phi(ts)\geq C_\phi \phi(t)\phi(s),\;\;\forall\, t, s >0\,.
\eeq
Let $q>2$, $\alpha=\frac 1 2 -\frac 1 q$. Then
\beq\label{non-hom-Sobolev-bis}
\|x\|_q\leq C(\|\psi(L)^{\alpha }x\|_2+ \|x\|_2)\;,\;\;\forall\; x\in \rm {Dom} (L),
\eeq
implies
\beq\label{local-uc-bis}
\|T_tx\|_\8\leq A \phi(t)^{-1}\|x\|_1,\;\forall t\in(0,1),\;\forall x\in L_1(\M),\eeq
where the constant $A$ depends only on $\phi$, $\alpha$ and the constant $C$ in \eqref{non-hom-Sobolev-bis}.
\end{prop}

\begin{proof}
Putting $T_1x$ in \eqref{non-hom-Sobolev-bis}, we have
\be
\|T_1x\|_q\leq C(\|\psi(L)^{\alpha }T_1 x\|_2+ \|x\|_2)\;.
\ee
Note that $\psi(L)^{\alpha }T_1= m(L)$ with $m(\lambda) = \psi(\lambda)^\alpha e^{-\lambda}$. Since $(\phi,\psi)$ is regularly relative,
\be
m(\lambda)& \approx e^{-\lambda} \big( \int _0^\8 e^{-\lambda t} \phi(t)^\alpha \frac{dt}{t}\big)^{-1}.
\ee
Lebesgue's monotone convergence theorem then ensures
$$\lim_{\lambda\ra 0} m(\lambda) \approx\big(\lim_{\lambda\ra 0} \int _0^\8 e^{-\lambda t} \phi(t)^\alpha \frac{dt}{t}\big)^{-1}<\8.$$
 On the other hand, by the $\Delta_2$ condition of $\phi,$ for $\lambda$ large enough,
\be\begin{split}
m(\lambda)&\approx e^{-\lambda} \bigg[ \int _0^\8 e^{- t} \phi\big(\frac t \lambda \big)^\alpha \frac{dt}{t}\bigg]^{-1}\\
&\leq e^{-\lambda}  \bigg[ \int _0^\8 e^{- t} \phi(t )^\alpha \lambda ^{-\alpha \log _2C_\phi} \frac{dt}{t}\bigg]^{-1}\\
&\leq \lambda ^{\alpha \log _2C_\phi} e^{-\lambda}  \bigg[ \int _0^\8 e^{- t} \phi(t )^\alpha \frac{dt}{t}\bigg]^{-1} \,.
\end{split}\ee
Then we know that $m(\cdot)$ is a bounded function on $\real^+$, whence $\|\psi(L)^{\alpha }T_1 x\|_2\leq C\|x\|_2.$ Thus we obtain
\beq\label{T1-2toq}
\|T_1x\|_q\leq C\|x\|_2\;.
\eeq

Moreover, putting $(I-T_1)^kx$ in \eqref{non-hom-Sobolev-bis}, where $k$ is a positive integer to be chosen later, we have
\beq\label{I-T1}
\begin{split}
\|(I-T_1)^kx\|_q &\leq C(\|\psi(L)^{\alpha }(I-T_1)^k x\|_2+ \|(I-T_1)^kx\|_2)\\
&\leq C(\|\psi(L)^{\alpha } x\|_2+ \|\psi(L)^{-\alpha }(I-T_1)^k\psi(L)^{\alpha }x\|_2)\;.
\end{split}\eeq
We will show that $\psi(L)^{-\alpha }(I-T_1)^k$ is bounded on $L_2(\M)$. Again, it suffices to show the boundedness of $m_1(\lambda) = \psi(\lambda)^{-\frac \alpha k} (1-e^{-\lambda})$ on $\real^+$. Since
$$\psi(\lambda)^{-\frac \alpha k}\approx \int _0^\8 e^{- t} \phi\big(\frac t \lambda \big)^{\frac \alpha k} \frac{dt}{t},$$
and $\phi$ is increasing, $m_1(\lambda)\leq C \int _0^\8 e^{- t} \phi(t)^{\frac \alpha k} \frac{dt}{t}$ for all $\lambda\geq 1$. For $\lambda \in (0,1)$, the $\Delta_2$ condition of $\phi$ ensures
$$\phi(\frac t \lambda)\leq C_\phi^{-\log_2 \lambda}\phi (t),$$
whence
$$\phi(\frac t \lambda)^{\frac \alpha k} \leq  \lambda ^{-\frac{\alpha \log_2 C_\phi }{k}} \phi (t)^{\frac \alpha k}.$$
Thus, if $k$ is chosen so that $\frac{C_\phi \alpha}{k}<1,$
$$m_1(\lambda)\leq C \lambda ^{-\frac{C_\phi \alpha}{k}} (1-e^{-\lambda}) \int _0^\8 e^{- t} \phi(t)^{\frac \alpha k} \frac{dt}{t}$$
is then bounded for $\lambda\in (0,1)$. This gives the boundedness of $\psi(L)^{-\alpha }(I-T_1)^k$ on $L_2(\M).$ Thus \eqref{I-T1} reads
\beq\label{I-T1-conclusion}
\|(I-T_1)^kx\|_q \leq C\|\psi(L)^{\alpha } x\|_2\;
\eeq
for chosen $k$. Note that this $k$ depends on $\phi$ and $\alpha$.

We now use \eqref{condition-phi}, \eqref{T1-2toq} and \eqref{I-T1-conclusion} to deduce \eqref{local-uc-bis}. Indeed, rewrite
$$T_t(I-T_1)^k = \psi (L)^{-\alpha} (I-T_1)^k \psi(L)^{\alpha}T_t \,.$$
Denote $m_t(L)=\psi(L)^{\alpha}T_t$ with $m_t(\lambda)=\psi(\lambda)^\alpha e^{-t\lambda } $. By hypothesis \eqref{condition-phi},
\be\begin{split}
\sup_{\lambda\in \real^+} m_t(\lambda)&\approx\sup_{\lambda\in \real^+} e^{-t\lambda} \big( \int _0^\8 e^{-\lambda s} \phi(s)^\alpha \frac{ds}{s}\big)^{-1}\\
&=\sup_{\lambda\in \real^+} e^{-\lambda} \big( \int _0^\8 e^{- s} \phi(\frac{st}{\lambda})^\alpha \frac{ds}{s}\big)^{-1}\\
&\leq  \phi(t)^{-\alpha}\sup_{\lambda\in \real^+} e^{-\lambda} \big( \int _0^\8 e^{- s} \phi(\frac{s}{\lambda})^\alpha \frac{ds}{s}\big)^{-1}\,.
\end{split}\ee
Here, the last supremum is bounded, following from the discussion in first part of this proof. Thus $|||m_t(L)|||_{2}\leq C \phi(t)^{-\alpha}$. Applying \eqref{I-T1-conclusion}, we have
$$|||T_t(I-T_1)^k|||_{2\ra q}\leq C \phi(t)^{-\alpha}\,.$$
It follows from \eqref{T1-2toq} that for $x\in L_2(\M)$,
\be\begin{split}
\|T_t x\|_q& \leq  \|T_t(I-T_1)^k x\|_q +\sum_{1\leq j\leq k}C_{j,k} \|T_{t+j} x\|_q\\
&\leq  C_{\phi,\alpha}(\phi(t)^{-\alpha}+1)\|x\|_2\,.
\end{split}\ee
Define a new semigroup $\widetilde T_t =e^{-t}T_t$. Then $\|\widetilde T_t x\|_q \leq C_{\phi,\alpha}\phi(t)^{-\alpha}\|x\|_2$. By Theorem \ref{phipsiultra}, we have
$$\|\widetilde T_t x\|_\8 \leq C_{\phi,\alpha}\phi(t)^{-1}\|x\|_1\,,$$
which gives \eqref{local-uc-bis}.

\end{proof}

To prove the converse implication, we need to introduce the subordinated semigroups. See \cite[Section~IX.11]{Yosida} for more details. The subordinated semigroup of order $\alpha \in (0,1)$ of $T_t$ is the semigroup
$$T_{t,\alpha} =\int_0^\8 f_{t,\alpha}(s)T_s ds\,,$$
where $f_{t,\alpha}$ is the function whose Laplace transform is the function $z\mapsto e^{-tz^\alpha}$ for $\rm{Re}(z)>0$. $(T_{t,\alpha})_{t>0}$ is then a semigroup, whose the infinitesimal generator is given by $$L^\alpha =\frac{\sin \alpha \pi}{\pi}\int_0^\8 s^{\alpha -1}(sI+L)^{-1}ds\,.$$
The function $f_{t,\alpha}$ is determined by the Post-Widder inversion formula
 $$f_{t,\alpha}(s) =\lim_{n\ra \8} \frac{(-1)^n}{n!}(\frac n s)^{n+1} \rho^{(n)}(\frac n s)\,,\;s>0$$
 where $\rho (a) =e^{-ta^\alpha}=\int_0^\8 e^{-sa} f_{t,\alpha}(s)ds.$  One can then deduces that $f_{t,\alpha}\geq 0$ and
$$\int_0^\8 f_{t,\alpha}(s)ds =1\,.$$
Whence
\beq\label{subordiT}
T_{t,\alpha} =\int_0^\8 f_{1,\alpha}(s)T_{st^{\frac 1 \alpha}}ds\,.\eeq
Moreover,
\beq\label{subordi-estimate}
\int_0^\8 s^{-n }f_{1,\alpha}(s)ds <\8,\;\;\forall\; n>0.\eeq

Now we are able to conclude the proof of Theorem \ref{local-thm}. The argument below is verbatim the same as in the commutative case. See \cite{VSCC1992}.

\begin{proof}[End of proof of Theorem \ref{local-thm}]
Firstly we assume that $\alpha \nu <1$. By \eqref{local-uc} and the contractivity of $T_t$, we have
$$|||T_t|||_{1\ra \8}\leq A\,,\;\;\forall \;t \geq 1.$$
Then \eqref{subordiT} and \eqref{subordi-estimate} ensure that for $0<t<1$
\be\begin{split}
|||T_{t,\alpha\nu}|||_{1\ra \8} &\leq \int_0^\8 f_{1,\alpha}(s)|||T_{st^{\frac 1{\alpha \nu}}}|||_{1\ra\8}ds \\
&\leq  t^{-\frac{1}{\alpha}}\int_0^{t^{-\frac 1{\alpha \nu}}}A s^{-\nu} f_{1,\alpha}(s)ds +\int_{t^{-\frac 1{\alpha \nu}}}^\8 A f_{1,\alpha}(s) ds \\
&\leq Ct^{-\frac{1}{\alpha}}.
\end{split}\ee
Consider the semigroup $\bar T_t =e^{-t}T_{t,\alpha\nu}$ with infinitesimal generator $I+L^{\alpha\nu}$. The above inequality gives
$$|||\bar T_t|||_{1\ra \8}\leq C t^{-\frac{1}{\alpha}},\;\;\forall t>0.$$
It then follows from Theorem \ref{phipsiultra} that $\forall\; x\in \rm {Dom} (L)$,
$$\|x\|_q\leq C \|(I+L^{\alpha \nu})^{\frac{1}{\alpha}\cdot \alpha}x\|_2 =C(\|L^{\alpha \nu}x\|_2+ \|x\|_2).$$

In general, if $\alpha \nu \geq 1$, choose an integer $k>\alpha\nu$. Applying the above argument to $T_{t,\frac{\alpha\nu}{k}}$, we then get
$$\|x\|_q\leq C \|(I+L^{\frac{\alpha \nu}k})^{\frac{k}{\alpha}\cdot \alpha}x\|_2.$$
Thus the theorem follows from induction and the fact that, for an operator $P$ on $L_2(\M)$
$$\|Px\|_2\leq C(\|x\|_2+\|P^2x\|_2),$$
which is ensured by the spectral theorem and the easy inequality $t\leq 1+t^2$ for $0<t<\8$.
\end{proof}

We end this section with a result of localization at infinity, parallel to Theorem \ref{local-thm}.

\begin{thm}
Assume that $|||T_1|||_{1\ra \8}<\8 $. Then the property
\beq\label{local-uc-infty}
\|T_tx\|_\8\leq A t^{-\nu}\|x\|_1,\;\forall t\in(1,\8),\;\forall x\in L_1(\M),\eeq
is equivalent to
\beq\label{local-uc-sb}
\|x\|_q\leq C(\|L^{\alpha \nu}x\|_2+ \|L^{\alpha \nu}x\|_q)\;,\;\;\forall\; x\in {\rm {Dom}} (L),
\eeq
where $q>2$, $\alpha=\frac 1 2 -\frac 1 q$.
\end{thm}

\begin{proof}

Firstly we prove the implication \eqref{local-uc-sb}$\Rightarrow$\eqref{local-uc-infty}. By interpolation, $|||T_1|||_{1\ra \8}<\8 $ implies $|||T_1|||_{2\ra q}<\8$, whence
$$\|L^{\alpha \nu}T_1 x\|_q\leq C \|L^{\alpha \nu} x\|_2.$$
\eqref{local-uc-sb} then gives
\be
\|T_1 x\|_q\leq C(\|L^{\alpha \nu}T_1 x\|_2+ \|L^{\alpha \nu}T_1x\|_q)\leq C\|L^{\alpha \nu} x\|_2.
\ee
Writing $T_{t+1} = T_1 L^{-\alpha \nu}L^{\alpha \nu} T_t$, we deduce from the last inequality that
\beq\label{norm-1+t}
|||T_{t+1} |||_{2\ra q} \leq C   |||T_1 L^{-\alpha \nu}|||_{2\ra q} |||L^{\alpha \nu} T_t |||_{2\ra 2}\leq C t^{-\alpha \nu}.
\eeq
Now for $r>1$ and $x\in L_1(\M)$, set
$$K(r, x) =\sup _{1<t<r} \frac{t ^{\nu (1-\frac 1 q)}\|T_t x\|_q}{\|x\|_1}.$$
If $1<t<r$, it follows from \eqref{norm-1+t} that
\be\begin{split}
\|T_{2t+1} x\|_q &\leq C t^{-\alpha \nu}\|T_t x\|_2\\
& \leq C t^{-\alpha \nu} \|T_tx \|_1^{1-\theta}\|T_t x\|_q^\theta\\
&\leq C t^{-\alpha \nu} \|T_tx \|_1^{1-\theta} t^{-\theta \nu (1-\frac 1 q )} \|x\|_1^\theta K(r,x)^\theta\\
& \leq C t^{- \nu(1-\frac 1 q )} \|x\|_1  K(r,x)^\theta,
\end{split}
\ee
where $\theta \in (0,1)$ so that $\frac 1 2 = 1-\theta +\frac \theta q$. Whence
\be\begin{split}
\sup _{3<t<2r+1} \frac{t ^{\nu (1-\frac 1 q)}\|T_{t} x\|_q}{\|x\|_1}&=\sup _{1<t<r} \frac{(2t+1) ^{\nu (1-\frac 1 q)}\|T_{2t+1} x\|_q}{\|x\|_1}\\
&\leq C\sup _{1<t<r} \frac{t ^{\nu (1-\frac 1 q)}\|T_{2t+1} x\|_q}{\|x\|_1}\leq C K(r,x)^\theta.
\end{split}\ee
On the other hand, since $|||T_1|||_{1\ra \8}<\8 $, we have a uniform constant $A$ such that
$$\sup _{1<t\leq 3} \frac{t ^{\nu (1-\frac 1 q)}\|T_{t} x\|_q}{\|x\|_1}<A.$$
Therefore, we can find a constant $C$ independent of $r$ and $x$ such that
$$ K(r,x)\leq  K(2r+1,x) =\sup _{1<t<2r+1} \frac{t ^{\nu (1-\frac 1 q)}\|T_{t} x\|_q}{\|x\|_1}\leq C K(r,x)^\theta.$$
Thus we have for every $t>1$,
$$\frac{t ^{\nu (1-\frac 1 q)}\|T_{t} x\|_q}{\|x\|_1}<C^{\frac 1 {1-\theta }}.$$
Then, to conclude the proof of \eqref{local-uc-infty}, we need only to use a dual argument starting from \eqref{norm-1+t}.

For the opposite, we first assume that $\alpha\nu <\beta $. Note that
$$L^{-\alpha \nu}T_1 x =\int_0^\8 t^{\alpha \nu} T_{t+1}x \frac{dt}{t}.$$
Then the proof of Lemma \ref{iimpliesiv} can be adapted to give
$$\|T_1 x\|_q \leq C \|L^{\alpha\nu}x\|_2.$$
Moreover, for $0<\beta<1$, the subordinate semigroup $(T_{t,\beta})_{t>0}$ satisfies
$$L^{-\alpha \nu}T_{1,\beta}\, x =\int_0^\8 t^{\frac {\alpha \nu}{\beta}} T_{t+1,\beta}\,x \frac{dt}{t},$$
whence also
$$\|T_{1,\beta} x\|_q \leq C \|L^{\alpha\nu}x\|_2.$$
It then follows that
$$
\|x\|_q\leq  \|T_{1,\beta}\,x\|_q+ \|(I-T_{1,\beta})x\|_q\leq C\|L^{\alpha\nu}x\|_2+ \|(I-T_{1,\beta})x\|_q.
$$
So it remains to establish the term $\|(I-T_{1,\beta})x\|_q$, which is controlled by
$$\|(I-T_{1,\beta})x\|_q \leq \int_0^1 \|L^\beta T_{t,\beta} x\|_q  dt\leq C\|L^{\alpha \nu}x\|_q.$$
The general case where $\alpha\nu \geq \beta $ can be treated by considering the power of $I-T_{1,\beta}$, similar to the last step of the proof of Theorem \ref{local-thm}.

\end{proof}

\n{\bf Acknowledgements.} The author is greatly indebted to professor Quanhua Xu for having suggested to him the subject of this paper, for many helpful discussion and very careful reading of this paper.


\end{document}